\documentclass[a4paper]{amsart}
\usepackage[
left=40truemm,
right=40truemm
]{geometry}
\usepackage{amsfonts, amssymb, amsmath, verbatim, amsthm, mathrsfs, amscd, enumerate, ascmac, fancyhdr, caption, hyperref}
\usepackage[dvipdfmx,svgnames]{xcolor}
\usepackage{tikz}
\usetikzlibrary{positioning,intersections, calc, arrows,decorations.markings}
\usepackage{pgfplots}
\theoremstyle{plain}
\newtheorem{theorem}{Theorem}
\newtheorem{lemma}[theorem]{Lemma}
\newtheorem{proposition}[theorem]{Proposition}
\newtheorem{corollary}[theorem]{Corollary}
\theoremstyle{definition}

\theoremstyle{remark}
\newtheorem*{remark}{Remark}

\def\supp{\textrm{supp}\,}
\def\d{\textrm{d}}
\def\ae{\textrm{a.e.}}
\author{Chu-hee Cho}
\address[Chu-hee Cho] {Department of Mathematical Sciences and RIM, Seoul National University, Seoul 151--747, Republic of Korea}
\email{akilus@snu.ac.kr}
\author{Shobu Shiraki}
\address[Shobu Shiraki] {Department of Mathematics, Graduate school of Science and Engineering, Saitama University, Saitama, 338-8570, Japan}
\email{s.shiraki.446@ms.saitama-u.ac.jp}
\thanks{The first author was supported by 
	NRF grant No. 2017R1D1A1A02019547 (Republic of Korea).}
\begin{document}
\date{\today}
\title[Pointwise convergence along a curve]{Pointwise convergence along a tangential curve for the fractional Schr\"odinger equation}
\begin{abstract}
In this paper we study the pointwise convergence problem along a tangential curve for the fractional Schr\"odinger equations in one spatial dimension and estimate the capacitary dimension of the divergence set. We extend a prior paper by Lee and the first author for the classical Schr\"odinger equation, which in itself contains a result due to Lee, Vargas and the first author, to the fractional Schr\"odinger equation. The proof is based on a decomposition argument without time localization, which has recently been introduced by the second author.
\end{abstract}
\maketitle
\section{Introduction}
Let $d\ge1$. We consider the fractional Schr\"odinger equation on $\mathbb R^d\times \mathbb R$
\[
\begin{cases}
	i\partial_t u + (-\Delta)^{\frac{m}2} u = 0, &\\
	u(\cdot,0) = f \in H^{s}(\mathbb R^d),  &	
\end{cases}
\]
with the initial data $f\in H^s(\mathbb R^d)$ and $m>0$. Here, $H^s(\mathbb R^d)$ denotes the Sobolev space of order $s$ whose norm is given by
\[
\|f\|_{H^s(\mathbb{R}^d)}=\|(1-\Delta)^\frac s2f\|_{L^2(\mathbb{R}^d)}. 
\]  
Then, the solution can be formally written as 
\[
u(x,t) = e^{it(-\Delta)^{\frac m2}}f(x) =\left( \frac1{2\pi}\right)^d \int_{\mathbb R^d} e^{i(x\cdot \xi + t|\xi|^m)} \widehat{f}(\xi)\,\d\xi,
\]
where $\widehat \cdot$ is the Fourier transform defined by $\widehat{f}(\xi)=\int_{\mathbb R^d} e^{-ix\cdot\xi} f(x)\,\d x$. While the classical and standard Schr\"odinger equation when $m=2$ has drawn attention in numerous fields, the fractional Sch\"odinger equation with general $m$ has also been found to be influential in recent years. In fact, it is not only a model case of a general dispersive equation \cite{CS88, KPV91} but also one of the fundamental equations in quantum mechanics \cite{Ls00, Ls02}. Since then, it has become an important subject studied by a number of authors from a variety of perspectives, see for example \cite{CKS16, CL14, GH11, GLNY18, GY14, HS15, IP14, Ke12, Ps09}. 

A fundamental question is the following pointwise convergence problem of determining the minimal exponent $s$ for which 
\begin{equation}\label{conv}
\lim_{t\to0} e^{it(-\Delta)^{\frac m2}}f(\gamma(x,t))=f(x)\quad\ae 
\end{equation}
for $f\in H^s(\mathbb R^d)$.
Here, $\gamma$ is a continuous function such that 
\[
\gamma: \mathbb R^d \times[-1,1]\to\mathbb R^d, \quad\gamma(x,0)=x.
\]
The key quality of $\gamma$ for \eqref{conv} is the way of convergence whether \textit{tangential} or \textit{non-tangential} to the hyperplane $\mathbb R^d\times\{0\}$.

The simplest example of $\gamma$ is $\gamma(x,t)=x$ and may be considered as the prototypical non-tangential case. For this $\gamma$, \eqref{conv} reduces to the seminal pointwise convergence problem, so-called Carleson's problem, originating in \cite{Cr80}. It turns out for $d=1$ and $m=2$ that \eqref{conv} holds for all $f\in H^s(\mathbb R)$ if and only if $s\ge\frac14$, due to \cite{Cr80} and Dahlberg and Kenig \cite{DK82}. Later, Sj\"olin \cite{Sj87} generalized their results and proved $s\ge\frac14$ is also necessary and sufficient even for $m>1$. In higher dimensions, $d\ge2$, the problem has attracted significant attention and been studied by many authors. When $m=2$, it has recently been understood that $s\ge\frac{d}{2(d+1)}$ is necessary for \eqref{conv} by Bourgain in \cite{Br16} and $s>\frac{d}{2(d+1)}$ is sufficient for \eqref{conv} by Du, Guth and Li \cite{DGL17} and Du and Zhang \cite{DZ18}. Ko and the first author \cite{CK18} also proved $s>\frac{d}{2(d+1)}$ is sufficient for \eqref{conv} when $m>1$ as well. The reader may also refer to in particular the work of Vega \cite{Vg88}, Lee \cite{Lee06}, Bourgain \cite{Br13}, and Du, Guth, Li and Zhang \cite{DGLZ18} as papers which have played an important role in earlier developments.

In the study of pointwise convergence problem for the Schr\"odinger equation with harmonic oscillator potential, Lee and Rogers \cite{LeeR12} 
showed that any $\gamma\in C^1(\mathbb R^d\times[-1,1]\to\mathbb R^d)$, such as $\gamma(x,t)=x-(t^\kappa,0,\cdots,0)$ with $\kappa\ge1$, is essentially equivalent to the vertical line in the context of pointwise convergence problem of $\eqref{conv}$.\\

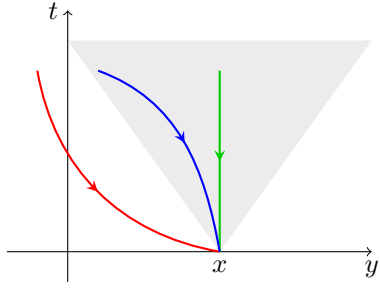
\begin{figure}[t]
\begin{center}
\begin{tikzpicture}[scale=0.8]

\tikzset{->-/.style={decoration={
  markings,
  mark=at position .5 with {\arrow{stealth}}},postaction={decorate}}}

  \fill [fill=lightgray!30!] (2.5,0)--(5,3.5)--(0,3.5)--(2.5,0);    

 \draw [->](-1,0)--(5,0);
 \draw [->](0,-0.5)--(0,4);

 \draw [color=red,thick,->-] (-0.5,3) to [out =-80, in=170] (2.5,0) ;
 \draw [color=green!80!black!, thick,->-](2.5,3)--(2.5,0);
  \draw [color=blue,thick, ->-] (0.5,3) to [out =-20, in=100] (2.5,0);
  
 \node [below]at (2.5,0){$x$};
 \node [below]at (5,0){$y$};
 \node [left]at (0,4){$t$}; 
 
\end{tikzpicture}
\caption{Three typical paths of convergence} \label{fig:curves}
\end{center}
\end{figure}

In contrast to the non-tangential case above, here we consider the tangential case. 
The curve $\gamma$ is said to satisfy H\"older condition of order $\kappa\in(0,1]$ in $t$ if
\begin{equation}\label{cond:Holder}
|\gamma(x,t)-\gamma(x,t')|\le C_1|t-t'|^{\kappa},\quad x\in\mathbb R^d, \quad t, t'\in[-1,1]
\end{equation}
and be bilipschitz in $x$ if
\begin{equation}\label{cond:bilipschitz}
\frac{1}{C_2}|x-x'|\le|\gamma(x,t)-\gamma(x',t)|\le C_2|x-x'|,\quad t\in[-1,1],\quad x, x'\in\mathbb R^d
\end{equation} 
for some $C_1, C_2 > 0$.
Then, we denote by $\Gamma(d,\kappa)$ the collection of such curves, namely,
\[
\Gamma(d,\kappa)=\{\gamma: \mathbb R^d\times[-1,1]\to\mathbb R^d:\text{$\gamma$ satisfies $\gamma(x,0)=x$, \eqref{cond:Holder} and \eqref{cond:bilipschitz}}\}.
\]
Note that $\Gamma(d,\kappa)$ contains $\gamma(x,t)=x-(t^\kappa,0,\cdots,0)$ with $0\le\kappa\le 1$. (See also Figure \ref{fig:curves}.) As a tangential case for $\gamma\in\Gamma(1,\kappa)$, Lee, Vargas and the first author \cite{CLV12} observed the crucial difference in nature from the non-tangential case; $s>\max\{\frac14,\frac{1-2\kappa}{2}\}$ is the sharp sufficient condition for \eqref{conv}.

We further consider a refinement of this question: quantifying  the sets on which the convergence \eqref{conv} fails in more precise way than Lebesgue measure. Let $0<\alpha\le d$. A positive Borel measure $\mu$ is said to be $\alpha$-dimensional if there exists a  constant $c$ such that
\begin{equation}\label{alpha-dimensional}
\mu(\mathbb B(x,r))\le c r^\alpha,
\end{equation}
where $\mathbb B(x,r)$ is the ball centered at $x\in\mathbb R^d$ with radius $r>0$. 
For $f \in H^{s}(\mathbb R^d)$ and $\gamma \in \Gamma(d,\kappa)$, let us define the \textit{divergence set} by 
\[
\mathfrak{D}(\gamma,f)=\{x\in \mathbb R^d:e^{it(-\Delta)^{m/2}}f(\gamma(x,t))\not\to f(x)\ {\rm as}\ t\to0\},
\]
and for a subset $X$ in $\mathbb R^d$, we also define the \textit{capacitary dimension} by
\[
\dim_c(X)=\sup\{\alpha: \mathcal M^\alpha(X)\not=\emptyset\},
\]
where $\mathcal M^\alpha(X)=\{\mu:\text{$\mu$ is $\alpha$-dimensional and $0<\mu(X)<\infty$}\}$. By the forthcoming Frostman's lemma, note that for a Borel set $X$ there exists $\mu\in\mathcal M^\alpha(X)$ if and only if the Hausdorff dimension of $X$ is greater than or equal to $\alpha$. In this case, the capacitary dimension of $X$ coincides with the Hausdorff dimension (see also \cite{Mt95}). Sj\"ogren and Sj\"olin \cite{SS89} considered this refined version of Carleson's problem for $\gamma(x,t)=x$ and $m\ge2$. Later, Barcel\'o, Bennett, Carbery and Rogers \cite{BBCR11} extended their result to $m>1$ and also obtained the sharp bound of the Hausdorff dimension of the divergence set by combining with work of \v Zubrini\'c \cite{Zb02}. In higher dimension, there are still some gaps remaining for which we refer the readers to contributions by Bennett and Rogers \cite{BR12}, Luc\`{a} and Rogers \cite{LR17, LR19b, LR19}, and aforementioned papers \cite{DGL17, DZ18, DGLZ18}. For $\gamma\in \Gamma(d,\kappa)$ and $m=2$, Lee and the first author obtained the capacitary dimension of the divergence set in \cite{CL14} as a refinement of aforementioned result by them with Vargas\footnote{
The result due to \cite{CLV12} coincides with the case $\alpha=1$ in \cite{CL14}. Consequently, 
\[
|\mathfrak{D}(\gamma,f)| = 0
\] for all $f \in H^{s}(\mathbb R)$ with $s>\max\left\{\frac14,\frac12-\kappa\right\}$, where $|\cdot|$ is the Lebesgue measure defined in $\mathbb R$. 
}.\\

Let $d=1$ and write $\Gamma(\kappa)=\Gamma(1,\kappa)$ in the rest of the paper. Our goal of the present note is to estimate the capacitary dimension of divergence sets for the pointwise convergence to the fractional Schr\"odinger equation along the curve $\gamma\in\Gamma(\kappa)$. Let us define evolution operator $S_t$ on appropriate input functions by
\[
S_tf(x)=\frac{1}{2\pi}\int_{\mathbb{R}}e^{i(x\xi+t|\xi|^m)}\widehat{f}(\xi)\,\d\xi.
\]
Our main results are the following.

\begin{theorem}\label{pointwise convergence with mu}
Let $m>1$, $0<\kappa\le1$, $\mu$ be an $\alpha$-dimensional measure and $\gamma\in\Gamma(\kappa)$. If $s>\max\left\{\frac14, \frac{1-\alpha}{2}, \frac{1-m\alpha\kappa}{2} \right\}$, then
\[
\lim_{t\to0}S_t(f(\gamma(x,t)))=f(x),\quad\text{$\mu$-$\ae\ x$}
\]
for all $f\in H^s(\mathbb{R})$.
\end{theorem}

By a standard argument, this is reduced to the following local maximal estimate. 

\begin{theorem}\label{thm:maximal estimate with mu}
Let $m>1$, $0<\kappa\le1$, $\mu$ be an $\alpha$-dimensional measure and $\gamma\in\Gamma(\kappa)$. If $s>\max\left\{\frac14, \frac{1-\alpha}{2}, \frac{1-m\alpha\kappa}{2} \right\}$, then there exists a constant $C$ such that
\begin{equation}\label{ineq:maximal estimate with mu}
\left(\int_{-1}^1\sup_{t\in [-1,1]}|S_tf(\gamma(\cdot,t))|^2\,\d\mu(x)\right)^\frac12\le C \|f\|_{H^s}
\end{equation}
for all $f\in H^s(\mathbb{R})$.
\end{theorem}

It is straightforward to obtain some results of the above type for $m\not=2$ (more specifically $m\ge2$) by appropriately modifying the argument in \cite{CLV12} or \cite{DN19}, however, as the second author observed in his study of a related problem in a different setting \cite{Sh19}, there are some barriers with such an approach to treat $m$ near $1$. In particular, building on the ideas in \cite{Sh19} in which we completely avoid time localization techniques, we are able to handle the full range of $m>1$ and give us the sharp sufficient conditions. Here, saying sharp is meant by that: Suppose $s<\max\left\{\frac14, \frac{1-\alpha}{2}, \frac{1-m\alpha\kappa}{2} \right\}$, then there exists $\gamma\in\Gamma(\kappa)$, $\alpha$-dimensional $\mu$ and $f\in H^s(\mathbb R)$ such that \eqref{pointwise convergence with mu} fails.
Since the counterexamples can be provided by adjusting the corresponding well-known constructions (for instance, \cite{CLV12} and \cite{Sj87}) without any major difficulty, we rather focus on the sufficient conditions.
As corollaries of Theorem \ref{pointwise convergence with mu}, we have the following.

\begin{corollary}\label{divergence set estimate}
Let $m>1$, $0<\kappa\le1$, $\gamma\in\Gamma(\kappa)$. If $s>\frac14$, then
\[
\dim_c(\mathfrak{D}(\gamma,f))\le\max\left\{1-2s,\frac{1-2s}{m\kappa}\right\}.
\]
\end{corollary}

The special case when $\mu$ is the ($1$-dimensional) Lebesgue measure extends the result in \cite{CLV12} from $m=2$ to $m>1$ as follows. Here, note that the required regularity on $s$ for \eqref{conv} depends not only on $\kappa$ but $m$ as well.

\begin{corollary}\label{pointwise convergence without mu}
Let $m>1$, $0<\kappa\le1$ and $\gamma\in\Gamma(\kappa)$.  If $s>\max\{\frac14,\frac{1-m\kappa}{2}\}$, then
\begin{equation}\label{ineq:maximal estimate without mu}
\left(\int_{-1}^1\sup_{t\in [-1,1]}|S_tf(\gamma(\cdot,t))|^2\,\d x\right)^\frac12\le C \|f\|_{H^s}
\end{equation}
holds for all $f\in H^s(\mathbb{R})$.
\end{corollary}

Ding and Niu have also considered \eqref{ineq:maximal estimate without mu} for $m\ge2$ in \cite{DN19} and claim that the sharp sufficient condition is $s>\max\{\frac14,\frac{1-2\kappa}{2}\}$ (in particular, their condition is independent of $m$). Unfortunately, it appears that the arguments in \cite{DN19} are not complete and the sharp regularity threshold is  $\max\{\frac14,\frac{1-m\kappa}{2}\}$; we note that the necessity of $s\ge\max\{\frac14,\frac{1-m\kappa}{2}\}$ for \eqref{ineq:maximal estimate without mu} follows from Section \ref{sec:counterexamples} in the present paper by taking $\alpha=1$. 

Combining Corollary \ref{pointwise convergence without mu} with the result from \cite{Sh19}, the results in \cite{CLV12} have been completely extended from the standard Schr\"odinger equation to the fractional Schr\"odinger equation with\footnote{
One can also consider the fractional Schr\"odinger equation with $0<m\le 1$ but the nature appears to be certainly different. For instance, see \cite{Cw82, RV08, Wl95} for $\gamma(x,t)=x$.
}
$m>1$. 

\begin{remark}
Although Theorem \ref{thm:maximal estimate with mu} is stated for the fractional Schr\"odinger evolution operator, by simply following the provided proof in Section \ref{sec:proof of maximal estimate} the same conclusion is valid for a wider class of evolution operators such as
\[
S_t^{\Phi}f(x)=\frac{1}{2\pi}\int_{\mathbb{R}}e^{i(x\xi+t\Phi(\xi))}\widehat{f}(\xi)\,\d\xi,
\]
where $\Phi:\mathbb{R}\to\mathbb{R}$ is a $C^2$-function for which there exist constants $C_3,C_4>0$ such that
\begin{equation}\label{Phi'' is bounded below}
|\xi|^{2-m}\left|\frac{\d^2}{\d\xi^2}\Phi(\xi)\right|\ge C_3
\quad\text{and}\quad
|\xi|\left|\frac{\d^2}{\d\xi^2}\Phi(\xi)\right|\ge C_4\left|\frac{\d}{\d\xi}\Phi(\xi)\right|
\end{equation}
for all $|\xi|\ge1$. This class trivially contains $|\xi|^m$ whenever $m>1$.
\end{remark}

Throughout the paper, we denote $I=[-1,1]$, $A\gtrsim B$ if $A\ge CB$, $A\lesssim B$ if $A\le CB$ and $A\sim B$ if $C^{-1}B\le A \le CB$ for some constant $C>0$. The domain of certain norms is sometimes abbreviated, for its meaning is clear from the context. In the following Section \ref{sec:preliminaries}, we present useful lemmas, then first prove Theorem \ref{pointwise convergence with mu} and Corollary \ref{divergence set estimate} in Section \ref{sec:reduction arguments}. In Section \ref{sec:proof of maximal estimate}, we prove Theorem \ref{thm:maximal estimate with mu} by employing the philosophy in \cite{Sh19} of a decomposition argument without time localization. Finally, in Section \ref{sec:counterexamples} we see the shaprness of Theorem \ref{thm:maximal estimate with mu}.

\section{Preliminaries}\label{sec:preliminaries}
In this section, as we have informed, we introduce useful lemmas which we use multiple times in the rest of the paper.
\begin{lemma}[Frostman's lemma]\label{lem:Frostman}
Let $d\ge1$ and $X$ be a Borel set in $\mathbb{R}^d$. Then, $\dim_c(X)\ge\alpha$ if and only if there exists $\alpha$-dimensional measure such that $\supp \mu\subset X$ and $0<\mu(\mathbb{R}^d)<\infty$. Further, $\mu(X)>0$.
\end{lemma}
For a proof of Lemma \ref{lem:Frostman}, we refer the reader to \cite{Mt95}.
We need the following lemmas in order to prove Theorem \ref{thm:maximal estimate with mu} in Section \ref{sec:proof of maximal estimate}.

\begin{lemma}[van der Corput's lemma]\label{lem:van der Corput}
Let $-\infty<a<b<\infty$, $\phi$ be a sufficiently smooth real-valued function and $\psi$ be a bounded smooth complex-valued function. Suppose we have $|\frac{\d^k}{\d \xi^k}\phi(\xi)|\geq1$ for all $\xi \in [a,b]$. If $k=1$ and $\frac{\d}{\d \xi}\phi(\xi)$ is monotonic on $(a,b)$, or simply $k\geq2$, then there exists a constant $C_k$ such that
\[
\left|\int_a^be^{i\lambda\phi(\xi)}\psi(\xi)\,\d\xi\right|\leq C_k\lambda^{-\frac1k}\left(\int_a^b\left|\frac{\d}{\d\xi}\psi(\xi)\right|\,\d\xi+\|\psi\|_{L^{\infty}}\right)
\]
for all $\lambda > 0$.
\end{lemma}
For a proof of Lemma \ref{lem:van der Corput}, we refer the reader to \cite{St94}.

\begin{lemma}\label{lem:Young and HLS-type}
Let $0<\alpha\le1$ and $\mu$ be an $\alpha$-dimensional measure. There exists a constant $C$ such that for any interval $[a,b]$ ($-\infty<a,b<\infty$) 
\begin{align}\label{ineq:Young special case}
&\left|\iint\iint g(x,t)h(x',t')\chi_{[a,b]}(x-x')\,\d \mu(x)\d t\d \mu(x')\d t'\right|\\
&\le C(b-a)^{\alpha}\|g\|_{L^2_x(\d\mu)L^1_t}\|h\|_{L^2_x(\d\mu)L^1_t}.\nonumber
\end{align}
Moreover, for $0<\rho<\alpha$ there exists a constant $C$ such that 
\begin{align}\label{ineq:HLS-type}
&\left|\iint\iint g(x,t)h(x',t')|x-x'|^{-\rho}\,\d \mu(x)\d t\d \mu(x')\d t'\right|\\
&\le C\|g\|_{L^2_x(\d\mu)L^1_t}\|h\|_{L^2_x(\d\mu)L^1_t}.\nonumber
\end{align}
Here, the both integrals are taken over $(x,t)$, $(x',t')\in I\times I$.
\end{lemma}
\begin{proof}[Proof of Lemma \ref{lem:Young and HLS-type}]
Denoting $G(x)=\|g(x,\cdot)\|_{L^1}$ and $H(x')=\|h(x',\cdot)\|_{L^1}$,
\begin{align*}
&\left|\iint\iint g(x,t)h(x',t')\chi_{[a,b]}(x-x')\,\d\mu(x)\d t\d\mu(x')\d t'\right|\\
&\le\int_{-1}^1\int_{-1}^1G(x)H(x')\chi_{[a,b]}(x-x')\d\mu(x)\d\mu(x').
\end{align*}
By invoking the Cauchy--Schwarz inequality on $L^2(I\times I,\d\mu\d\mu)$ and \eqref{alpha-dimensional},
\begin{align*}
&\int_{-1}^1\int_{-1}^1G(x)H(x')\chi_{[a,b]}(x-x')\d\mu(x)\d\mu(x')\\
&\lesssim
\left(\iint G(x)^2\chi_{[a,b]}(x-x')\,\d\mu(x)\d\mu(x')\right)^\frac12
\left(\iint H(x')^2\chi_{[a,b]}(x-x')\,\d\mu(x)\d\mu(x')\right)^\frac12\\
&
\lesssim(b-a)^{\alpha}\|G\|_{L^2_x(\d\mu)}\|H\|_{L^2_x(\d\mu)}.
\end{align*}

Now, \eqref{ineq:HLS-type} follows from \eqref{ineq:Young special case}, immediately. In fact, by applying a dyadic decomposition,
\begin{align*}
&\left|\iint\iint g(x,t)h(x',t')|x-x'|^{-\rho}\,\d \mu(x)\d t\d \mu(x')\d t'\right|\\
&\lesssim\sum_{j=0}^\infty2^{\rho j}\iint G(x)H(x')\chi_{[2^{-j},2^{-j+1}]}(x-x')\,\d\mu(x)\d\mu(x')\\
&\lesssim\sum_{j=0}^\infty 2^{(\rho-\alpha)j}\|G\|_{L^2_x(\d\mu)}\|H\|_{L^2_x(\d\mu)}\\
&\lesssim\|G\|_{L^2_x(\d\mu)}\|H\|_{L^2_x(\d\mu)}
\end{align*}
whenever $\rho-\alpha<0$.
\end{proof}

\section{Some reduction arguments}\label{sec:reduction arguments}
\subsection{Proof of (Theorem \ref{thm:maximal estimate with mu} $\Longrightarrow$ Theorem \ref{pointwise convergence with mu})}
Fix an arbitrary $f\in H^s(\mathbb R)$. Then, it is enough to show that $\mu(\mathfrak D(\gamma,f))=0$. Now, choose a sequence $\{f_n\}_{n\in\mathbb N}\subset C_0^\infty(\mathbb R)$ which converges in $H^s$-norm to $f\in H^s(\mathbb R)$. Then, we divide the divergence set into localized pieces as follows and show that all terms turn out to be $0$.
\begin{align}\label{mu div-set is less than sum sum}
\mu(\mathfrak D(\gamma,f))
\le
\sum_{j\in\mathbb Z}\sum_{\ell=1}^\infty\mu(\{x\in I+j\,:\,\lim_{t\to0}|S_tf(\gamma(x,t))-f(x)|>\ell^{-1}\}).\nonumber
\end{align}
Now, for each $n \ge 1$, $j=0$ and $\lambda\ge1$ observe that
\begin{align*}
&\mu(\{x\in I\, :\, \lim_{t\to0}|S_tf(\gamma(x,t))-f(x)|>\lambda^{-1}\})\\
&\le
\mu(\{x\in I\, :\, \limsup_{t\to0}|S_tf(\gamma(x,t))-S_tf_n(\gamma(x,t))|>(3\lambda)^{-1}\})\\
&\qquad+
\mu(\{x\in I\, :\, \limsup_{t\to0}|S_tf_n(\gamma(x,t))-f_n(x)|>(3\lambda)^{-1}\})\\
&\qquad\quad+
\mu(\{x\in I\, :\, |f_n(x)-f(x)|>(3\lambda)^{-1}\})\\
&\le
\mu(\{x\in I\,:\,\sup_{t\in I}|S_t\big(f(\gamma(x,t))-f_n(\gamma(x,t))\big)|>(3\lambda)^{-1}\})\\
&\qquad+
0
+
\mu(\{x\in I\, :\, |f_n(x)-f(x)|>(3\lambda)^{-1}\}).
\end{align*}
By invoking the Chebyshev's inequality and Theorem \ref{thm:maximal estimate with mu} we obtain
\begin{equation}\label{mu div-set is less than f-f H2}
\mu(\{x\in I\, :\, \lim_{t\to0}|S_tf(\gamma(x,t))-f(x)|>\lambda^{-1}\})\lesssim\lambda^2\|f-f_n\|_{H^s(\mathbb R)}^2,
\end{equation}
which tends to $0$ as $n\to\infty$. For other $j$, make translation $x\mapsto x+j$ and we define a measure $\mu_j$ by $\mu_j(x) = \mu(x+j)$ and a curve $\gamma_j$ by $\gamma_j(x,t) = \gamma(x+j,t)$, both of which satisfy the required conditions\footnote{
Strictly speaking, $\gamma_j \not\in \Gamma(\kappa)$ because $\gamma_j(x,0)=x+j$, however this translation effect is negligible for Theorem \ref{thm:maximal estimate with mu} since $|\gamma_j(x,t)-\gamma_j(x',t')|$ is essentially equivalent to $|\gamma(x,t)-\gamma(x',t')|$ for any $x,x',t,t'$.
}
for Theorem \ref{thm:maximal estimate with mu} so that \eqref{mu div-set is less than f-f H2} holds with $I$ replaced by $I+j$. Therefore, for all $j\in\mathbb Z$ and $\ell\ge1$ 
\[
\mu(\{x\in I+j\,:\, \lim_{t\to0}|S_tf(\gamma(x,t))-f(x)|>\ell^{-1}\})=0
\]
holds as desired. \qed

\subsection{Proof of (Theorem \ref{pointwise convergence with mu} $\Longrightarrow$ Corollary \ref{divergence set estimate})}
Let $s>\frac14$ and $f\in H^s(\mathbb R)$. If we suppose $\dim_c(\mathfrak{D}(f,\gamma))>\max\{1-2s,\frac{1-2s}{m\kappa}\}\ge0$, then one would find $0<\alpha<1$ satisfying $\dim_c(\mathfrak{D}(f,\gamma))>\alpha>\max\{1-2s,\frac{1-2s}{m\kappa}\}\ge0$. Here, note that the second inequality is equivalent to $s>\max\{\frac{1-\alpha}{2},\frac{1-m\alpha\kappa}{2}\}$. Hence, by Lemma \ref{lem:Frostman} there would exist an $\alpha$-dimensional measure $\mu$ such that $\mu(\mathfrak{D}(f,\gamma))>0$, which contradicts Theorem \ref{pointwise convergence with mu}, and we must have $\dim_c(\mathfrak{D}(f,\gamma))\le\max\{1-2s,\frac{1-2s}{m\kappa}\}$. \qed

\section{Proof of Theorem \ref{thm:maximal estimate with mu}}\label{sec:proof of maximal estimate}

Let
\[
s_*=\min\left\{\frac14,\frac\alpha2,\frac{m\alpha\kappa}{2}\right\}.
\]
By following the standard steps via Littlewood--Paley decomposition, it is enough to show the following proposition. (For the details, for instance, see  \cite{Sh19}.)
\begin{proposition}\label{prop:Sf is less than lambda f with mu}
Let $\varepsilon>0$. Then, there exists a constant $C_\varepsilon$ such that
\begin{equation}\label{ineq:SPf is less than f with mu}
\left\|\sup_{t\in I}|S_tf(\gamma(\cdot,t))|\right\|_{L^2(I,\d\mu)}\le C_\varepsilon \lambda^{\frac12-s_*+\varepsilon}\|f\|_{L^2}
\end{equation}
holds for all $\lambda\ge1$ and $f\in L^2(\mathbb{R})$ whose Fourier support is contained in $\{\xi\in\mathbb{R}:\frac{\lambda}{2}\le|\xi|\le2\lambda\}$.
\end{proposition}
\begin{proof}[Proof of Proposition \ref{prop:Sf is less than lambda f with mu}]

Let
\[
Tf(x,t)=\chi(x,t)\int_{\mathbb{R}}e^{i(\gamma(x,t)\xi+t|\xi|^m)}f(\xi)\psi(\tfrac\xi\lambda)\,\d\xi,
\]
where $\chi=\chi_{I\times I}$ and $\psi\in C_0^\infty((-2,-\frac12)\cup(\frac12,2))$. Then, by Plancherel's theorem, \eqref{ineq:SPf is less than f with mu} follows from
\begin{equation}\label{ineq:Tf is less than f}
\|Tf\|_{L^2_x(\d\mu)L^\infty_t}^2\lesssim\lambda^{1-2s_*+\varepsilon}\|f\|_{L^2}^2.
\end{equation}
By duality, \eqref{ineq:Tf is less than f} is equivalent to
\begin{equation}\label{ineq:TF is less than F}
\|T^*F\|_{L^2}^2\lesssim\lambda^{1-2s_*+\varepsilon}\|F\|_{L^2_x(\d\mu)L^1_t}^2,
\end{equation}
where
\[
T^*F(\xi)=\psi(\tfrac{\xi}{\lambda})\iint \chi(x',t')e^{-i(\gamma(x',t')\xi+t'|\xi|^m)}F(x',t')\,\d\mu(x')\d t'.
\]
Then,
\begin{align*}
\|T^*F\|_{L^2}^2&=\int\psi(\tfrac{\xi}{\lambda})^2\iint\iint\chi(x,t)\chi(x',t')\\
&\qquad\times e^{i((\gamma(x,t)-\gamma(x',t'))\xi+(t-t')|\xi|^m)}\overline{F}(x,t)F(x',t')\,\d\mu(x)\d t\d\mu(x')\d t'\d\xi\\
&=\int_W\int_{W'}\chi(w)\chi(w')\overline{F}(w)F(w')K_\lambda(w,w')\,\d_\mu w\d_\mu w'\\
&=\sum_{\ell=1}^3\iint_{V_\ell}\chi(w)\chi(w')\overline{F}(w)F(w')K_\lambda(w,w')\,\d_\mu w\d_\mu w'\\
&=\mathcal I_1+\mathcal I_2+\mathcal I_3.
\end{align*}
Here, we denote $W=I\times I$, $w=(x,t)\in W$, $w'\in(x',t')\in W$ and $\d_\mu w=\d\mu(x)\d t$. Also,
\begin{align*}
K_\lambda(w,w')&=\int_\mathbb{R} e^{i\phi(\xi,w,w')}\psi(\tfrac{\xi}{\lambda})^2\,\d\xi\\
&=\lambda\int_\mathbb{R} e^{i\phi(\lambda\xi,w,w')}\psi(\xi)^2\,\d\xi,
\end{align*}
\[
\phi(\xi,w,w')=(\gamma(x,t)-\gamma(x',t'))\xi+(t-t')|\xi|^m
\]
and
\[
\begin{cases}
V_1=\{(w,w')\in W\times W:|x-x'|\le2\lambda^{-\frac{2s_*}{\alpha}}\},\\
V_2=\{(w,w')\in W\times W:|x-x'|>2\lambda^{-\frac{2s_*}{\alpha}}\ {\rm and}\ \frac1{C_2}|x-x'|> 2C_1|t-t'|^\kappa\},\\
V_3=\{(w,w')\in W\times W: |x-x'|>2\lambda^{-\frac{2s_*}{\alpha}}\ {\rm and}\  \frac1{C_2}|x-x'|\le2C_1|t-t'|^\kappa\}.
\end{cases}
\]
Then, \eqref{ineq:TF is less than F} follows from 
\[
\mathcal I_\ell\lesssim\lambda^{1-2s_*+\varepsilon}\|F\|_{L^2_x(\d\mu)L^1_t}^2
\]
for each $\ell=1,2,3$.

\subsection*{The term $\mathcal I_1$} 
By using the trivial estimate 
\begin{equation}\label{ineq:trivial estimate}
|K_\lambda(w,w')|\lesssim\lambda
\end{equation}
and Lemma \ref{lem:Young and HLS-type}, we obtain
\begin{align*}
\mathcal I_1\lesssim\lambda^{1-2s_*}\|F\|_{L^2_x(\d\mu)L^1_t}^2.
\end{align*}

\subsection*{The term $\mathcal I_2$} 
In this case, first note that we have
\begin{equation}\label{ineq:gamma-gamma is less than x-x}
|\gamma(w)-\gamma(w')|\ge\frac1{2C_2}|x-x'|
\end{equation}
for $(w,w')\in V_2$ by using \eqref{cond:Holder} and \eqref{cond:bilipschitz}. Next, we split $K_\lambda$ into $\mathcal K_1$ and $\mathcal K_2$ as follows.
\begin{align*}
K_\lambda(w,w')&=\lambda\int_{U_1}e^{i\phi(\lambda\xi,w,w')}\psi(\xi)^2\,\d\xi+\lambda\int_{U_2}e^{i\phi(\lambda\xi,w,w')}\psi(\xi)^2\,\d\xi\\
&=:\mathcal K_1+\mathcal K_2,
\end{align*}
where 
\[
\begin{cases}
U_1=\{\xi\in\supp\psi:\frac1{C_2}|x-x'|>4m\lambda^{m-1}|t-t'||\xi|^{m-1}\},\\
U_2=\{\xi\in\supp\psi:\frac1{C_2}|x-x'|\le 4m\lambda^{m-1}|t-t'||\xi|^{m-1}\}.
\end{cases}
\]
For $\mathcal K_1$,
we use \eqref{ineq:gamma-gamma is less than x-x} in order to estimate the phase
\begin{align*}
\left|\frac{\d}{\d\xi}\phi(\lambda\xi,w,w')\right|&\ge\lambda|\gamma(w)-\gamma(w')|-m\lambda^m|t-t'||\xi|^{m-1}\\
&\ge\frac1{2C_2}\lambda|x-x'|-m\lambda^m|t-t'||\xi|^{m-1}\\
&>\frac1{4C_2}\lambda|x-x'|\\
&\gtrsim\lambda^{1-\frac{2s_*}{\alpha}}\\
&>1
\end{align*}
since $\frac{2s_*}{\alpha}=\min\{\frac{1}{2\alpha},1,m\kappa\}\le1$. Here, note that the interval $U_1$ consists of at most two intervals since $\frac{\d}{\d\xi}\phi(\lambda\xi,w,w')$ is monotone on each interval $(-\infty,-1]$ and $[1,\infty)$. Thus, Lemma \ref{lem:van der Corput} gives that 
\begin{align*}
\mathcal K_1\lesssim\lambda(\lambda|x-x'|)^{-1}\lesssim\lambda(\lambda|x-x'|)^{-\min\{\frac12,\alpha\}}.
\end{align*}

On the other hand, for $\mathcal K_2$, 
\begin{align*}
\left|\frac{\d^2}{\d\xi^2}\phi(\lambda\xi,w,w')\right|&\gtrsim\lambda^m|t-t'||\xi|^{m-2}\\
&\gtrsim\lambda|x-x'|
\end{align*}
so that we are allowed to apply Lemma \ref{lem:van der Corput} to obtain
\begin{align*}
\mathcal K_2&\lesssim\lambda(\lambda|x-x'|)^{-\frac12}\\
&\lesssim\lambda(\lambda|x-x'|)^{-\min\{\frac12,\alpha\}}.
\end{align*}
Hence, for $(w,w')\in V_2$ we have the following kernel estimate
\begin{align*}
|K_\lambda(w,w')|&\lesssim\lambda^{1-\min\{\frac12,\alpha\}}|x-x'|^{-\min\{\frac12,\alpha\}}\\
&\lesssim\lambda^{1-\min\{\frac12,\alpha\}+\varepsilon}|x-x'|^{-\min\{\frac12,\alpha\}+\varepsilon}.
\end{align*}
Here, we used the separation assumption, $|x-x'|\gtrsim\lambda^{-\frac{2s_*}{\alpha}}$.
By Lemma \ref{lem:Young and HLS-type} with $\rho=\min\{\frac12,\alpha\}-\varepsilon<\alpha$ we conclude that
\begin{align*}
\mathcal I_2&\lesssim\lambda^{1-\min\{\frac12,\alpha\}+\varepsilon}\|F\|_{L^2_x(\d\mu)L^1_t}^2\\
&\lesssim\lambda^{1-2s_*+\varepsilon}\|F\|_{L^2_x(\d\mu)L^1_t}^2.
\end{align*}

\subsection*{The term $\mathcal I_3$}
 It remains to consider $\mathcal I_3$. Observe that
\begin{align*}
\left|\frac{\d^2}{\d\xi^2}\phi(\lambda\xi,w,w')\right|&\gtrsim\lambda^m|t-t'||\xi|^{m-2}\\
&\gtrsim\lambda^m|x-x'|^\frac1\kappa\\
&\gtrsim\lambda^m\lambda^{-\frac{2s_*}{\alpha\kappa}}\\
&>1
\end{align*}
since $\frac{2s_*}{\alpha\kappa}=\min\{\frac{1}{2\alpha\kappa},\frac{1}{\kappa},m\}\le m$. Then, by Lemma \ref{lem:van der Corput} for arbitrary small $\varepsilon>0$
\begin{align*}
|K_\lambda(w,w')|&\lesssim\lambda(\lambda^m|x-x'|^\frac1\kappa)^{-\frac12}\\
&\lesssim\lambda(\lambda^m|x-x'|^{\frac1\kappa})^{-\frac{2s_*}{m}}\\
&\sim\lambda^{1-2s_*}|x-x'|^{-\frac{2s_*}{m\kappa}}\\
&\lesssim\lambda^{1-2s_*+\varepsilon}|x-x'|^{-\frac{2s_*}{m\kappa}+\varepsilon}
\end{align*}
since $\frac{2s_*}{m}=\min\{\frac{1}{2m},\frac{\alpha}{m},\alpha\kappa\}<\frac12$ and our separation assumption again. Therefore, applying Lemma \ref{lem:Young and HLS-type} with $\rho= \frac{2s_*}{m\kappa}-\varepsilon=\min\{\frac{1}{2m\kappa},\frac{\alpha}{m\kappa},\alpha\}-\varepsilon<\alpha$, it follows that
\[
\mathcal I_3\lesssim\lambda^{1-2s_*+\varepsilon}\|F\|_{L^2_x(\d\mu)L^1_t}^2.
\]

\end{proof}

\section{The necessary conditions regarding Theorem \ref{thm:maximal estimate with mu}}\label{sec:counterexamples}
In this section, we present $s\ge\max\left\{\frac14, \frac{1-\alpha}{2}, \frac{1-m\alpha\kappa}{2} \right\}$ is necessary for Theorem \ref{thm:maximal estimate with mu}, otherwise there exist $\gamma\in\Gamma(\kappa)$ and $\alpha$-dimensional measure $\mu$ such that \eqref{ineq:maximal estimate with mu} fails.
Throughout the section, we shall let $\lambda\ge1$, $\gamma(x,t)=x-t^\kappa$, $\mu(x)=|x|^{-1+\alpha}\,\d x$ and $\psi_0$ be a smooth radial bump function whose support is in a small neighborhood of the origin. Also, we fix $m>1$ and $0<\kappa\le1$, and we assume that the maximal estimate \eqref{ineq:maximal estimate with mu} holds.

\subsection*{The necessity of $s\ge\max\{\frac{1-\alpha}{2},\frac{1-m\alpha\kappa}{2}\}$}
In this case, we will simply follow the idea in \cite{CLV12}.
Let
\[
\widehat{f_1}(\xi)=\psi_0(\lambda^{-\frac1m}\xi).
\]
With this initial data, 
\begin{align*}
|S_tf_1(\gamma(x,t))|&\sim\left|\int e^{i((x-t^\kappa)\xi+t|\xi|^m)}\widehat{f_1}(\xi)\,\d\xi\right|\\
&=\lambda^{\frac1m}\left|\int e^{i\phi_1(\eta,x,t)}\psi_0(\eta)\,\d\eta\right|,
\end{align*}
where \[
\phi_1(\eta,x,t)=\lambda^{\frac1m}(x-t^\kappa)\eta+\lambda t|\eta|^m.
\]
For
$x\in(0,\frac{1}{100}(\lambda^{-\frac1m}+\lambda^{-\kappa}))$ and $|t|<\frac{1}{100}\lambda^{-1}$, we have
\[
|\phi_1(\eta,x,t)|=|\lambda^{\frac1m}(x-t^\kappa)\eta+\lambda t|\eta|^m)|\le\frac12
\]
so that
\begin{align*}
|S_tf_1(\gamma(x,t))|
&\gtrsim
\lambda^{\frac1m}\left|\int(\cos\phi_1(\eta,x,t))\psi_0(\eta)\,\d\eta\right|\\
&\gtrsim
\lambda^\frac1m\chi_{(0,\frac{1}{100}(\lambda^{-\frac1m}+\lambda^{-\kappa}))\times(0,\frac{1}{100}\lambda^{-1})}(x,t).
\end{align*}
Hence,
\begin{align*}
\left\|\sup_{t\in I}|S_tf_1(\gamma(\cdot,t))|\right\|_{L^2(I,\d\mu)}
&\ge
\left\|\sup_{t\in(0,\frac{1}{100}\lambda^{-1})}|S_tf(\gamma(\cdot,t))|\right\|_{L^2((0,\frac{1}{100}(\lambda^{-\frac1m}+\lambda^{-\kappa})),\d\mu)}\\
&\gtrsim\lambda^\frac1m\max\{\lambda^{-\frac\alpha m},\lambda^{-\alpha\kappa}\}^{\frac12}.
\end{align*}
On the other hand,
\begin{align*}
\|f_1\|_{H^s}&\sim\left(\int(1+|\xi|^2)^s|\psi_0(\lambda^{-\frac1m}\xi)|^2\,\d\xi\right)^\frac12\\
&\lesssim\lambda^{\frac sm}\lambda^{\frac{1}{2m}}.
\end{align*}
Therefore, combining the above calculations, we obtain
\[
\lambda^\frac1m\max\{\lambda^{-\frac\alpha m},\lambda^{-\alpha\kappa}\}^{\frac12}\lesssim\lambda^{\frac sm}\lambda^{\frac{1}{2m}}.
\]

If $\frac1m\ge\kappa$, then $\max\{\lambda^{-\frac\alpha m},\lambda^{-\alpha\kappa}\}=\lambda^{-\alpha\kappa}$, and as $\lambda\to\infty$ it is necessary that
\[
\frac 1m-\frac{\alpha\kappa}{2}\le\frac sm+\frac{1}{2m},
\]
which clearly gives
\[
s\ge\frac{1-m\alpha\kappa}{2}.
\]

If $\frac1m<\kappa$, then $\max\{\lambda^{-\frac\alpha m},\lambda^{-\alpha\kappa}\}=\lambda^{-\frac\alpha m}$, and as $\lambda\to\infty$, it is necessary that 
\[
\frac1m-\frac{\alpha}{2m}\le\frac sm+\frac{1}{2m},
\]
which is
\[
s\ge\frac{1-\alpha}{2}.
\]

\subsection*{The necessity of $s\ge\frac14$}
In this case, we will refer to the idea in \cite{Sj87} (see page 712). Let
\[
\widehat{f_2}(\xi)=\lambda^{-1}\psi_0(\lambda^{-1}\xi+\lambda).
\]
Then, by the change of variables $-\eta=\lambda^{-1}\xi+\lambda$, 
\begin{align*}
|S_tf_2(\gamma(x,t))|&\sim\left|\int e^{i((x-t^\kappa)\xi+t|\xi|^m)}\lambda^{-1}\psi_0(\lambda^{-1}\xi+\lambda)\,\d\xi \right|\\
&=\left|\int e^{i\phi_2(\eta,x,t)}\psi_0(-\eta)\,\d\eta\right|,
\end{align*}
where
\[
\phi_2(\eta,x,t)=-(x-t^\kappa)\lambda\eta+\lambda^m t|\lambda+\eta|^m.
\]
By a Taylor expansion,
\begin{align*}
(\lambda+\eta)^m&=\lambda^m(1+\lambda^{-1}\eta)^m\\
&=\lambda^m\left(1+m\lambda^{-1}\eta+\frac{m(m-1)}{2}\lambda^{-2}\eta^2+O(\lambda^{-3}|\eta|^3)\right)\\
&=\lambda^m+m\lambda^{-(1-m)}\eta+\frac{m(m-1)}{2}\lambda^{-(2-m)}\eta^2+O(\lambda^{-(3-m)}|\eta|^3),
\end{align*}
and it follows that
\begin{align*}
\phi_2(\eta,x,t)
&=
-\lambda x\eta+\lambda t^\kappa\eta+\lambda^{2m}t+m\lambda^{-(1-2m)}t\eta\\
&\qquad+
\frac{m(m-1)}{2}\lambda^{-(2-2m)}t\eta^2+O(\lambda^{-(3-2m)}t|\eta|^3)\\
&=
\lambda^{2m}t+\lambda (-x+t^\kappa+m\lambda^{-(2-2m)}t)\eta\\
&\qquad+
\frac{m(m-1)}{2}\lambda^{-(2-2m)}t\eta^2+O(\lambda^{-(3-2m)}t|\eta|^3).
\end{align*}
For $x\in(0,\frac{1}{100})$, we can choose $t(x)$ such that $x=t(x)^\kappa+m\lambda^{-(2-2m)}t(x)$. In fact, if we consider the function $\tau(t) = t^{\kappa} +m\lambda^{-(2-2m)}t$, then $\tau:[0,\infty)\to[0,\infty)$ is a strictly increasing bijection and 
$$0=\tau^{-1}(0)<\tau^{-1}(x)=t(x)<\tau^{-1}(\tfrac{1}{100})<\frac{\lambda^{2-2m}}{100m}.$$
Therefore, for such choice of $(x,t(x))$, it follows that
\[
|\phi_2(\eta,x,t(x))-\lambda^{2m}t(x)|\lesssim 0+\frac{1}{100}+O(\lambda^{-1})\le\frac12,
\]
which implies that
\begin{align*}
|S_tf_2(\gamma(x,t(x)))|
&\sim
\left|\int\cos(\phi_2(\eta,x,t(x))-\lambda^{2m}t(x))\psi_0(-\eta)\,\d\eta\right|\\
&\gtrsim
\chi_{(0,\frac{1}{100})}(x).
\end{align*}
Hence,
\[
\left\|\sup_{t\in I}|S_tf_2(\gamma(\cdot,t))|\right\|_{L^2(I,\d\mu)}
\gtrsim
1.
\]
On the other hand,
\begin{align*}
\|f_2\|_{H^s}&=\left(\int(1+|\xi|^2)^s|\lambda^{-1}\psi_0(\lambda^{-1}\xi+\lambda)|^2\,\d\xi\right)^\frac12\\
&\lesssim\lambda^{2s}\lambda^{-\frac12}.
\end{align*}
Therefore, combining the calculations above implies that
\[
1\lesssim \lambda^{2s}\lambda^{-\frac12}.
\]
As $\lambda\to\infty$, it is necessary that 
\[
s\ge\frac14.
\]
This ends the proof that $s\ge\max\left\{\frac14, \frac{1-\alpha}{2}, \frac{1-m\alpha\kappa}{2} \right\}$ is necessary for \eqref{ineq:maximal estimate with mu} to hold.


\subsection*{Acknowledgment} The authors would like to thank (the second author's adviser) Neal Bez for the constant encouragement and constructive suggestions. The second author also wishes to express his great appreciation to Sanghyuk Lee for introducing him to the problem.


\end{document}